\documentclass[review]{elsarticle}

\usepackage{lineno,hyperref}

\usepackage{amssymb}
\usepackage{amsthm}
\usepackage{amsmath}

\usepackage{adjustbox}
\usepackage{empheq}

\usepackage{tikz,xcolor}
\usepackage{pgfplots}
\pgfplotsset{compat=newest}
\usepackage{mathtools}
\usepackage{mathrsfs}
\usepackage{enumitem}

\usepackage{graphicx}
\usepackage{subcaption}
\usepackage{algorithm}
\usepackage{algpseudocode}
\usepackage{footnote}

\usepackage{hyperref}
\usepackage[margin=20mm]{geometry}

\pgfmathsetmacro{\width} {7cm}
\pgfmathsetmacro{\height} {7cm}

	\definecolor{DarkBlue}{rgb}{0.00,0.00,0.55}
	\definecolor{Black}{rgb}{0.00,0.00,0.00}
	\hypersetup{
		linkcolor = DarkBlue,
		anchorcolor = DarkBlue,
		citecolor = DarkBlue,
		filecolor = DarkBlue,
		colorlinks  = true,
		urlcolor = Black
	}

\modulolinenumbers[5]

\journal{Journal of \LaTeX\ Templates}

\definecolor{forestgreen}{rgb}{0.13, 0.55, 0.13}

\newtheorem{theorem}{Theorem}[section]

\newtheorem{corollary}[theorem]{Corollary}
\theoremstyle{definition}

\theoremstyle{remark}
\newtheorem{remark}{Remark}[section]

\usepackage{todonotes}
\definecolor{myblue}{RGB}{135, 206, 250}
\definecolor{myred}{RGB}{255,66,0}

\begin{document}

\begin{frontmatter}

\title{Exploiting Kronecker structure in exponential integrators: fast approximation of the action of $\varphi$-functions of matrices via quadrature}

\author[Oden]{Matteo Croci}
\address[Oden]{Oden Institute for Computational Engineering and Sciences,\\ The University of Texas at Austin, USA}

\author[BCAM,Oden]{Judit Mu\~noz-Matute}
\address[BCAM]{Basque Center for Applied Mathematics (BCAM), Bilbao, Spain}

\begin{abstract}
In this article, we propose an algorithm for approximating the action of $\varphi-$functions of matrices against vectors, which is a key operation in exponential time integrators. In particular, we consider matrices with Kronecker sum structure, which arise from problems admitting a tensor product representation. The method is based on quadrature approximations of the integral form of the $\varphi-$functions combined with a scaling and modified squaring method. Owing to the Kronecker sum representation, only actions of 1D matrix exponentials are needed at each quadrature node and assembly of the full matrix can be avoided.
Additionally, we derive \emph{a priori} bounds for the quadrature error, which show that, as expected by classical theory, the rate of convergence of our method is supergeometric. Guided by our analysis, we construct a fast and robust method for estimating the optimal scaling factor and number of quadrature nodes that minimizes the total cost for a prescribed error tolerance.
We investigate the performance of our algorithm by solving several linear and semilinear time-dependent problems in 2D and 3D. The results show that our method is accurate and orders of magnitude faster than the current state-of-the-art. 
\end{abstract}


\end{frontmatter}

\section{Introduction}

Exponential time integrators \cite{hochbruck2010exponential,hochbruck1998exponential,hochbruck2009exponential,cox2002exponential} are a class of methods for solving stiff semilinear systems of Ordinary Differential Equations (ODEs) of the form $u'(t)+Au(t)=f(t,u(t))$,
where $A$ is a square matrix and $f$ is a nonlinear function. Classical time integration schemes have an exponential scheme counterpart including exponential Runge-Kutta methods \cite{hochbruck2005explicit,kassam2005fourth}, exponential multistep methods \cite{hochbruck2011exponential} or exponential splitting schemes \cite{hansen2009exponential}, among many others. Exponential time-stepping methods incorporate the exact propagator of the homogeneous equation so that linear stability is satisfied by construction. However, such an advantage comes at the cost of having to compute $\varphi$-functions of the matrix $A$. These $\varphi$-functions are defined in terms of integrals of the exponential of $A$ times a polynomial and appear in all exponential integrators. 

The first strategies employed in exponential integrators are based on approximating the whole $\varphi$-function of $A$ \cite{berland2007expint,higham2005scaling,higham2009scaling} (a dense matrix in general), and are thus expensive in terms of both CPU time and memory. In the last decade, new research focused on instead computing the \emph{action} of $\varphi$-functions against a vector \cite{higham2020catalogue}, which is considerably more efficient whenever the matrix $A$ is sparse. Current approaches include rational Padé approximations \cite{al2010new,fasi2019arbitrary}, Krylov subspace methods \cite{gaudreault2018kiops,niesen2012algorithm}, and truncated Taylor series expansion \cite{al2011computing}. These new developments led to the application of exponential integrators in a wide range of applications \cite{crouseilles2020exponential,wang2020exponential,lord2013stochastic}.  

In most applications of exponential integrators the matrix $A$ and the resulting system of ODEs come from the semidiscretization in space of transient Partial Differential Equations (PDEs). In the specific case in which the spatial domain is a box and the coefficients of the PDE are separable, spatial discretizations such as Finite Differences (FD) or Finite Elements (FE) on tensor product grids or Isogeometric Analysis (IGA) \cite{palitta2016matrix} typically lead to a matrix $A$ with Kronecker sum structure, i.e.~$A=A^x\oplus A^y=A^x\otimes I^y+I^x\otimes A^y$ (in 2D). Here $A^{x,y}$ are 1D matrices arising from spatial discretization of the linear operator in a single spatial direction.

It is well known that the exponential of a matrix with Kronecker sum structure is equal to the Kronecker product of the exponentials of the one-dimensional matrices, and this property has been exploited in the literature to design efficient routines for computing matrix exponentials. For instance, in \cite{caliari2022mu} the authors propose an efficient CPU and GPU implementation of the exponential of Kronecker sums of matrices  for problems in arbitrary dimensions. Their method makes the solution of transient linear problems with zero source with exponential integrators extremely efficient, but it does not extend to more general semilinear problems. In fact, $\varphi$-functions of Kronecker sums do not simply separate into the Kronecker product of $\varphi$-functions of 1D matrices, making the tensor structure of the problem difficult to exploit in exponential integrators. Authors in \cite{munoz2022exploiting} recently proposed an algorithm that circumvents this problem by building on recurrence relations between $\varphi$-functions to recast the evaluation problem in terms of the action of 1D $\varphi$-matrix-functions. However, this algorithm does not generalize easily to the 3D case and is numerically unstable for high-order exponential integrators.


In this paper, we make the following new contributions:
\begin{itemize}
    \item We introduce a new method based on approximating the integral definition of the $\varphi$-functions via both fixed-point and adaptive quadrature (Gauss-Legendre and Clenshaw-Curtis respectively). Our algorithm inherits the numerical stability of quadrature rules and computations at each node are trivially parallelizable and only involve standard matrix exponentials. For this reason, only 1D matrix exponential actions are needed and no assembly of the full matrix $A$ is required.
    \item We provide an \emph{a priori} error analysis for our algorithm that builds on classical and modern theory on scalar quadrature methods \cite{rabinowitz1969rough,trefethen2008gauss,trefethen2019approximation}, and shows that our method converges at a supergeometric rate with respect to the number of nodes. Since our estimate grows exponentially with $\lVert A \rVert_{\infty}$, we combine our method with the scaling and modified squaring strategy from \cite{skaflestad2009scaling} to reduce the size of $\lVert A \rVert_{\infty}$.
    \item We design an algorithm for estimating the optimal scaling factor and number of quadrature nodes of the fixed-point quadrature strategy that minimizes the total cost while satisfying a given error tolerance. This algorithm is based on our theory and essentially only involves scalar and polynomial rootfinding operations which nowadays are robust and efficient numerical procedures. Our adaptive algorithm employs the same estimation routine for the optimal scaling factor, but then adaptively determines the number of nodes required.
\end{itemize}
We test the performance of our method in several linear and semilinear time-dependent problems and we conclude that, for matrices with Kronecker sum structure, our algorithm is accurate and order of magnitudes faster than the generic-purpose state-of-the-art routine from \cite{al2011computing}.

The article is organized as follows: Section \ref{sec:background} introduces the background needed, including the definition and properties of $\varphi$-functions and matrices with Kronecker sum structure, and exponential integrators. In Section \ref{sec:algorithm} we present and analyze our algorithm. We derive an \emph{a priori} quadrature error bound and present a routine for estimating the optimal scaling factor and number of quadrature nodes. In Section \ref{sec:numerics} we study the performance of our method for different 2D and 3D time-dependent linear and semilinear problems. Finally, we summarize our findings in Section \ref{sec:coclusions} and discuss suggestions for future work on the topic.

\section{Background}\label{sec:background}
We first recall the definition of $\varphi$-functions, exponential Runge-Kutta methods and matrices with Kronecker sum structure. 
\subsection{$\varphi$-functions and exponential time integrators}
In this paper we consider the following semilinear system of ODEs as model problem:
\begin{equation}\label{ODEs}
  \displaystyle{ \left\{
      \begin{aligned}
        u'(t)+Au(t)&=f(t,u(t)),\;\;\forall t\in (0,T],\\
        u(0)&=u_0,
      \end{aligned}
    \right.} 
\end{equation}
where A is a square matrix and f is a nonlinear term. Exponential integrators are constructed from different approximations of the integral form of the solution of system \eqref{ODEs}, the \textit{variation-of-constants formula}
\begin{equation}\label{VoC}
u(t)=e^{-tA}u_0+\int_{0}^te^{-(t-s)A}f(s,u(s))\ ds.
\end{equation}
This representation includes the exact propagator of the homogeneous equation (i.e.~for $f=0$) and different approximations of the source term in (\ref{VoC}) lead to different methods.  

The form of expression \eqref{VoC} leads to all exponential integrators being built in terms of the so-called $\varphi-$functions. After defining $\varphi_0(A):=e^{A}$, these are 
\begin{equation}\label{Varphi}
\varphi_p(A):=\int_0^1e^{(1-\theta)A}\frac{\theta^{p-1}}{(p-1)!}\ d\theta,\;\;\forall p\geq 1.
\end{equation}
The $\varphi$-functions satisfy the following recurrence relation
\begin{equation}\label{recurrence}
\varphi_p(A)=A\varphi_{p+1}(A)+\frac{1}{p!}I.
\end{equation}

For the time discretization of \eqref{VoC} with exponential integrators, we consider a uniform partition of the time interval 
$$0=t_0< t_1<\ldots<t_{m-1}<t_m=T,$$
with time step size $\tau=t_{k+1}-t_{k},\;\;\forall k=0,\ldots,m-1$. The simplest first-order exponential Runge-Kutta method is the exponential Euler method
$$u^{k+1}=u^{k}+\tau\varphi_1(-\tau A)\left(f(t_{k},u^{k})-Au^{k}\right),$$
which involves only $\varphi_1$.
This method is obtained by approximating the source term in (\ref{VoC}) by the constant value $f(t_{k},u^{k})$ and employing recurrence formula (\ref{recurrence}). More generally, s-stage exponential Runge-Kutta methods are given by 
\begin{equation}\label{eq:RK}
\displaystyle{ \left\{
\begin{split}
  u^{k+1}&=u^{k}+\tau\sum_{i=1}^{s}b(-\tau A)\left(f(t_k+c_i\tau,\ U^{ki})-Au^{k}\right),\\
  U^{ki}&=u^{k}+\tau\sum_{j=1}^{s}a_{ij}(-\tau A)\left(f(t_k+c_j\tau,\ U^{kj})-Au^{k}\right),\;\;\forall i=1,\ldots,s.
\end{split}
\right.} 
\end{equation}
Here, the coefficients $b_i$ and $a_{ij}$ are expressed in terms of linear combinations of $\varphi-$functions of the matrix $A$. As for traditional Runge-Kutta methods, the coefficients defining the methods \eqref{eq:RK} can be expressed via Butcher tableaus. We refer to \cite{hochbruck2010exponential} for an extensive review of existing methods and their properties. 

\subsection{Kronecker sum structure}
System \eqref{ODEs} often arises from a semi-discretization in space of transient Partial Differential Equations (PDEs). Here, we focus on the specific case in which the matrix $A$ has Kronecker sum structure, i.e.

\begin{equation}\label{KronA}
  \displaystyle{ \left\{
      \begin{split}
        &A=A^x\oplus A^y = A^x\otimes I^y+I^x\otimes A^y\;\;\;(2D),\\
        &A=A^x\oplus A^y \oplus A^z = A^x\otimes I^y\otimes I^z+I^x\otimes A^y\otimes I^z+I^x\otimes I^y\otimes A^z\;\;\;(3D).
      \end{split}
    \right.} 
\end{equation}
Here, $\oplus$ denotes the Kronecker sum and $\otimes$ denotes the Kronecker product, $I^{x,y,z}$ are one-dimensional identity matrices and $A^{x,y,z}$ are the matrices coming from the semidiscretization in each space direction.

The Kronecker sum structure \eqref{KronA} is obtained whenever the PDE has a tensor-product structure: the domain is a box, the PDE coefficients are separable, and the PDE is semidiscretized in space employing Finite Differences (FD), Finite Elements (FE) on tensor product grids, or Isogeometric Analysis (IGA) (see \cite{palitta2016matrix} for details).


It is well known \cite{benzi2017approximation} that the exponential of a matrix with Kroncker sum structure (\ref{KronA}) satisfies the following property
\begin{equation}\label{Expkron}
e^{A}=e^{A^x}\otimes e^{A^y}\;(2D),\;\;\; e^{A}=e^{A^x}\otimes e^{A^y}\otimes e^{A^z}\;(3D).
\end{equation}
A crucial ingredient of the algorithm we propose in the next section is a routine to compute matrix-vector product with $e^A$ efficiently. For this purpose, we exploit the following relations:
\begin{equation}\label{KronAb}
  \displaystyle{ \left\{
      \begin{split}
        &v=(e^{A^x}\otimes e^{A^y})b\Longleftrightarrow V=e^{Ay}Be^{(A^x)^T}\;\;\;(2D),\\
        &v=(e^{A^x}\otimes e^{A^y}\otimes e^{A^z})b\Longleftrightarrow V=B\times_1e^{A^x}\times_2e^{A^y}\times_3e^{A^z}\;\;\;(3D),
      \end{split}
    \right.} 
\end{equation}
where $b=\mbox{vec}(B)$, $v=\mbox{vec}(V)$ and $\mbox{vec}(\cdot)$ is the vectorization operator. In the 2D case in (\ref{KronAb}), V and B are matrices while in 3D they are tensors of order 3. Here we are indicating with $\times_{d}$ with $d\in\{1,2,3\}$ the Tucker operator. Performing matrix-vector products with the exponential as in \eqref{KronAb} is extremely efficient as it only involves dense linear algebra operations with 1D exponential matrices and can be accelerated on GPUs if needed \cite{caliari2022mu}.  We refer to \cite{caliari2022mublas} for a detailed presentation on multidimensional tensor algebra and its efficient implementation.

\begin{remark}
In this article, we only consider 2D and 3D time-dependent PDEs. However, the second equivalence in \eqref{KronAb} holds for matrices with Kronecker sum structure in arbitrary dimensions $d$
$$A=A^1\oplus A^2\oplus\ldots\oplus A^d.$$
While the extension of our algorithm to dimensions higher that $3$ is straightforward, we work in 2D and 3D in this paper for simplicity. 
\end{remark}

\section{New Algorithm}\label{sec:algorithm}
In this section we introduce our algorithm for approximating the action of $\varphi$-functions of matrices. Our method is based on numerical quadrature (both adaptive and fixed-point) combined with a scaling and modified squaring approach. In what follows we also provide an \textit{a priori} error estimate for the quadrature error and we design a robust and efficient strategy for computing the optimal scaling factor and number of quadrature nodes that minimizes costs for a given error tolerance.

\subsection{Approximation of $\varphi$-functions via quadrature}
The relations \eqref{KronAb} lead to an efficient algorithm for computing the action of the matrix exponential. However, \eqref{KronAb} is a direct consequence of property \eqref{Expkron}, which does not hold for the $\varphi$-functions. Our objective is to obtain an efficient algorithm for evaluating actions of $\varphi_p(A)$ for $p>0$ that can still exploit the Kronecker structure in $A$ without performing any full matrix assembly. For this purpose, we rely on equation \eqref{Varphi} to express the action of any $\varphi$-function of a matrix against a vector $b$ as

\begin{align}
\varphi_p(A)b=\int_0^1\frac{\theta^{p-1}}{(p-1)!}e^{(1-\theta)A}b\ d\theta = I_p,\;\;\forall p\geq 1.
\end{align}
Since the above is just a one-dimensional integral of an \emph{analytic} function over a bounded interval, we can approximate it via any suitable $(n+1)$-point 1D quadrature rule:
\begin{align}
\varphi_p(A)b\approx \sum_{i=1}^{n+1} w_i\frac{x_i^{p-1}}{(p-1)!}e^{(1-x_i)A}b = \hat{I}_p,\;\;\forall p\geq 1 ,
\end{align}
where $\{(w_i,x_i)\}_{i=1}^{n+1}$ are the quadrature weights and nodes and the action of the matrix exponential at the nodes can be computed efficiently via \eqref{KronAb}. While any geometrically convergent quadrature scheme is suitable for this purpose, we mainly employ Gauss-Legendre or Clenshaw-Curtis quadrature as they both come with sharp error bounds \cite{trefethen2019approximation} that we can leverage in our analysis. While Gaussian quadrature is more accurate, Clenshaw-Curtis is a nested rule and can therefore be used adaptively with live error estimation and automatic selection of the number of nodes required to achieve a prescribed tolerance. In Section \ref{sec:numerics} we study and compare the performance of both approaches in numerical experiments.

Employing a quadrature rule has three advantages: 1) It converges supergeometrically fast (see next subsection) so only a few matrix-vector products with the exponential are needed. 2) The integrand values at different nodes can be evaluated independently in parallel. 3) The same quadrature rule (and the same matrix-vector products with $e^{(1-x_i)A}$) can be used to compute the actions $\varphi_j(A)b$ for all $j=1,\dots,p$ at the same time with little extra cost. We present our method in Algorithm \ref{alg:1} (fixed-point quadrature version) and in Algorithm \ref{alg:2} (adaptive version).

\begin{algorithm}[h!]
\caption{Fixed-point quadrature algorithm for computing $\varphi_j(A)b$ for $j=1,\dots,p$.}
\label{alg:1}
\begin{enumerate}[leftmargin=*,align=left]
    \item[\textbf{Input:}] An integer $p$, a vector $b$, the matrices $A^{x,y,z}$, and a quadrature rule $\{(w_i,x_i)\}_{i=1}^{n+1}$.
    \begin{itemize}
        \item Compute and store the vectors $v_i=e^{(1-x_i)A}b$ for $i=1,\dots,n+1$ using \eqref{KronAb}.
        \item Compute the vectors $y_j = \sum_{i=1}^n w_i\dfrac{x_i^{j-1}}{(j-1)!}v_i$ for $j=1,\dots,p$.
    \end{itemize}
    \item[\textbf{Output:}] The products $y_j=\varphi_j(A)b$ for $j=1,\dots,p$.
\end{enumerate}
\end{algorithm}

\begin{algorithm}[h!]
\caption{Adaptive quadrature algorithm for computing $\varphi_j(A)b$ for $j=1,\dots,p$.}
\label{alg:2}
\begin{enumerate}[leftmargin=*,align=left]
    \item[\textbf{Input:}] An integer $p$, a vector $b$, the matrices $A^{x,y,z}$, and a relative error tolerance $\varepsilon$.
    \begin{itemize}
        \item Set $n=3$, $\text{err}=\infty$. Run Algorithm \ref{alg:1} with the $(2n+1)$-point Clenshaw-Curtis quadrature rule $\{(w_i,x_i)\}_{i=1}^{2n+1}$ and obtain the approximations $y_j$ for $j=1,\dots,p$, as well as the vectors $v_i$ for $i=1,\dots,2n+1$.
        \item \textbf{While} $\text{err} > \varepsilon$:
    \begin{enumerate}[label=\arabic*)]
        \item Set $n=2n$, $\tilde{y}_j=y_j$ for $j=1,\dots,p$, and $\tilde{v}_i=v_i$ for $i=1,\dots,n+1$. Construct the Clenshaw-Curtis $(2n+1)$-point rule $\{(w_i,x_i)\}_{i=1}^{2n+1}$. Note that the nodes $x_{2i-1}$ for $i=1,\dots,n+1$ and the nodes of the previously constructed $(n+1)$-point rule coincide.
        \item Compute the vectors $v_{2i}=e^{(1-x_{2i})A}b$ for $i=1,\dots,n$ and set $v_{2i-1}=\tilde{v}_i$ for $i=1,\dots,n+1$.
        \item Compute the vectors $\displaystyle{y_j = \sum_{i=1}^{2n+1} w_i\dfrac{x_i^{j-1}}{(j-1)!}v_i}$ for $j=1,\dots,p$.
        \item Update the error: $\text{err} = \max_j \dfrac{\lVert y_j - \tilde{y}_j \rVert_{\infty}}{\lVert y_j \rVert_{\infty}}$.
    \end{enumerate}
    \end{itemize}
    \item[\textbf{Output:}] The products $y_j=\varphi_j(A)b$ for $j=1,\dots,p$.
\end{enumerate}
\end{algorithm}

\begin{remark}
Linear combinations between the actions of different $\varphi$-functions against different vectors can also be computed efficiently as
\begin{align}
    \sum_{j=1}^p\varphi_j(A)b_j = \int_0^1e^{(1-\theta)A}\sum_{j=1}^p\frac{\theta^{j-1}}{(j-1)!}b_j\ d\theta \approx \sum_{i=1}^{n+1} w_ie^{(1-x_i)A}\sum_{j=1}^p\frac{x_i^{j-1}}{(j-1)!}b_j,
\end{align}
where $b_1,\dots,b_p$ are arbitrary vectors. However, we were unable to make this strategy compatible with the generalized scaling and squaring technique of Section \ref{sec:scaling}.
\end{remark}

\subsection{Error bounds}
We now focus on Algorithm \ref{alg:1} only for simplicity, and derive a bound for the quadrature error. For this purpose, we need the following result by Trefethen \cite{trefethen2019approximation}:
\begin{theorem}[Theorem 19.3 in \cite{trefethen2019approximation}]
\label{thm:Trefethen}
Let $E_\rho$ be an open Bernstein ellipse (i.e.~an ellipse with foci at $\pm 1$) with $\rho$ being the sum of its semimajor and semiminor axis lengths. Let a function $f$ be analytic in $[-1,1]$ and analytically continuable to $E_{\rho}$, where it satisfies $|f(z)|\leq M$ for some $M$. Then, $(n+1)$-point Clenshaw-Curtis quadrature with $n\geq 4$ applied to $\displaystyle{I=\int_{-1}^1f(x) \text{ d}x}$ satisfies
\begin{align}
    \label{eq:clenshaw-curtis-bound}
    |I-\hat{I}|\leq \frac{144}{35}\frac{M\rho^{1-n}}{\rho^2-1}.
\end{align}
Here $\hat{I}$ denotes the approximate integral. Furthermore, $(n+1)$-point Gaussian quadrature with $n\geq 2$ satisfies
\begin{align}
    \label{eq:gaussian-quadrature-bound}
    |I-\hat{I}|\leq \frac{144}{35}\frac{M\rho^{-2n}}{\rho^2-1}.
\end{align}
The factor $\rho^{1-n}$ in \eqref{eq:clenshaw-curtis-bound} can be improved to $\rho^{-n}$ if $n$ is even.
\end{theorem}

We now employ Theorem \ref{thm:Trefethen} to obtain an error bound for Algorithm \ref{alg:1}. The result is stated in the following theorem and corollary.
\begin{theorem}
\label{thm:error_bound}
For any integer $p\geq 0$, let $I_{p+1}=\varphi_{p+1}(A)b$, and let $\hat{I}_{p+1}$ be the approximation of $I_{p+1}$ computed via Algorithm \ref{alg:1} with a total of $n+1$ quadrature nodes. Then, provided that
$n\geq 4$ for Clenshaw-Curtis quadrature and $n\geq 2$ for Gaussian quadrature, we have that
\begin{align}
    \label{eq:thm_bnd_CC}
    \lVert I_{p+1} - \hat{I}_{p+1} \rVert_{\infty} &\leq \frac{144}{35}\frac{M(\bar{\rho})\bar{\rho}^{1-n}}{\bar{\rho}^2-1},\quad\text{for Clenshaw-Curtis quadrature},\\
    \lVert I_{p+1} - \hat{I}_{p+1} \rVert_{\infty} &\leq \frac{144}{35}\frac{M(\bar{\rho
    })\bar{\rho}^{\hspace{1pt}-2n}}{\bar{\rho}^2-1},\quad\text{for Gaussian quadrature},
    \label{eq:thm_bnd_Gauss}
\end{align}
where $M(\rho)$ is given by
\begin{align}
    M(\rho)= \frac{g(\rho)^p}{2^{p+1}p!}e^{\frac{1}{2}g(\rho)\lVert A \rVert_{\infty}} \lVert b \rVert_{\infty},\quad\text{where}\quad g(\rho)=\frac{(\rho+1)^2}{2\rho},
\end{align}
and $\bar{\rho}$ satisfies $\bar{\rho}>1$ and is a real root of the monic polynomial equation 
\begin{align}
    \label{eq:thm_polyeq}
    \rho^4 + a_3\rho^3+a_2\rho^2+a_1\rho+1 = 0,
\end{align}
whose coefficients are given by
\begin{align}
\label{eq:polycoeffs}
\begin{array}{llll}
    a_3 = - 4\lVert A \rVert_{\infty}^{-1}(n+1-p),& a_2 = -(2+8p\lVert A \rVert_{\infty}^{-1}),& a_1 =  4\lVert A \rVert_{\infty}^{-1}(n-1+p)&\text{Clenshaw-Curtis},\\
    a_3 = - 4\lVert A \rVert_{\infty}^{-1}(2(n+1)-p),& a_2 = -(2+8p\lVert A \rVert_{\infty}^{-1}),& a_1 = 4\lVert A \rVert_{\infty}^{-1}(2n+p),&\text{Gaussian quadrature}.
\end{array}
\end{align}
If $n$ is even we can replace $n$ with $n+1$ for Clenshaw-Curtis quadrature.
\end{theorem}

\begin{corollary}
\label{coroll:error_bound}
Under the same assumptions of Theorem \ref{thm:error_bound}, if we further assume that
$n\geq\max\left(4,\lceil\frac{1+\sqrt{2}}{2}\lVert A \rVert_{\infty} + p\right\rceil)$, $(n+1)$-point Clenshaw-Curtis quadrature with $n$ even yields an error of
\begin{align}
    \label{eq:thm_clenshaw_curtis_result}
    \lVert I_{p+1} - \hat{I}_{p+1} \rVert_{\infty} \leq \frac{\lVert b \rVert_{\infty}}{2^{p+1}p!}\left(\frac{n-p}{\frac{e}{2}\lVert A \rVert_{\infty}}\right)^{-(n-p)}.
\end{align}
If $n$ is odd, the bound still holds with $n$ replaced by $n-1$.
Provided that $n\geq\max\left(2,\lceil\frac{1+\sqrt{2}}{4}\lVert A \rVert_{\infty} + \frac{p}{2}\rceil\right)$, $(n+1)$-point Gaussian quadrature instead gives an error of
\begin{align}
    \label{eq:thm_gaussian_result}
     \lVert I_{p+1} - \hat{I}_{p+1} \rVert_{\infty} \leq \frac{\lVert b \rVert_{\infty}}{2^{p+1}p!}\left(\frac{2n-p}{\frac{e}{2}\lVert A \rVert_{\infty}}\right)^{-(2n-p)}.
\end{align}
\end{corollary}
\begin{proof}
We prove both Theorem \ref{thm:error_bound} and Corollary \ref{coroll:error_bound} for Gaussian quadrature only since the result for Clenshaw-Curtis quadrature follows the same argument. In order to apply Theorem \ref{thm:Trefethen}, the first step is to map the integral in \eqref{Varphi} onto $[-1,1]$
\begin{align}
    I_{p+1} = \varphi_{p+1}(A)b = \int_{-1}^1 \frac{(s+1)^p}{2^{p+1}p!}e^{\frac{1}{2}(1-s)A}b\text{ d}s=\int_{-1}^1G(s)b\text{ d}s.
\end{align}
The second step is to provide an upper bound for the module of the integrand in $E_\rho$. Since the integrand is vector-valued, we work with the infinity norm to provide an upper bound for all its entries and bound $\lVert G(s)b \rVert_{\infty}$ over $E_\rho$. We have that
\begin{align}
    \lVert G \rVert_{\infty}\leq \max_{s\in E_\rho} \frac{|s+1|^p}{2^{p+1}p!}e^{\frac{1}{2}|1-s|\ \lVert A \rVert_{\infty}}.
\end{align}
Since the maximum of $|1\pm s|$ over $E_{\rho}$ is attained on the rightmost or leftmost points of the ellipse at which $s=\pm\frac{\rho^2+1}{2\rho}$, we get $|1\pm s|\leq \frac{(\rho+1)^2}{2\rho}=g(\rho)$, and a bound for $\lVert G(s)b \rVert_{\infty}$ of
\begin{align}
\label{eq:Gs-bound}
   \lVert G(s)b \rVert_{\infty}\leq \frac{g(\rho)^p}{2^{p+1}p!}e^{\frac{1}{2}g(\rho)\lVert A \rVert_{\infty}} \lVert b \rVert_{\infty} = M(\rho).
\end{align}
Applying Theorem \ref{thm:Trefethen} to each entry of the integrand we obtain that for Gaussian quadrature
\begin{align}
    \label{eq:thm_proof0}
    \lVert I_{p+1} - \hat{I}_{p+1} \rVert_{\infty} \leq \frac{144}{35}\frac{M(\rho)\rho^{-2n}}{\rho^2-1} = \frac{144}{35}\frac{\rho^{-2n}}{\rho^2-1}\frac{g(\rho)^p}{2^{p+1}p!}e^{\frac{1}{2}g(\rho)\lVert A \rVert_{\infty}} \lVert b \rVert_{\infty},\quad\forall\rho>1.
\end{align}
Minimizing with respect to $\rho$ for fixed $p$ we get that
\begin{align*}
    \bar{\rho} = \arg\min_{\rho>1}\frac{144}{35}\frac{\rho^{-2n}}{\rho^2-1}\frac{g(\rho)^p}{2^{p+1}p!}e^{\frac{1}{2}g(\rho)\lVert A \rVert_{\infty}} \lVert b \rVert_{\infty}=\arg\min_{\rho>1}\left(-2n\log(\rho)-\log(\rho^2-1)+p\log(g(\rho))+\frac{1}{2}g(\rho)\lVert A \rVert_{\infty} \right).
\end{align*}
Differentiating the expression in the large brackets with respect to $\rho$ and setting the derivative to zero yields the polynomial equation $P(\rho)=\rho^4 + a_3\rho^3+a_2\rho^2+a_1\rho+1 = 0$ with coefficients
\begin{align}
a_3 = - 4\lVert A \rVert_{\infty}^{-1}(2(n+1)-p),\quad a_2 = -(2+8p\lVert A \rVert_{\infty}^{-1}),\quad a_1 = 4\lVert A \rVert_{\infty}^{-1}(2n+p).
\end{align}
Writing $\rho=1+x$ for $x\in\mathbb{C}$ and applying Descartes' rule of signs to the shifted polynomial
\begin{align}
    Q(x) = P(1+x) = x^4 + (4+a_3)x^3 + (3a_3+a_2+6)x^2  -8\lVert A \rVert_{\infty}^{-1}(2n+3)x - 8\lVert A \rVert_{\infty}^{-1},
\end{align}
it can be verified that the coefficients of $Q(x)$ change sign either once or three times depending on the values of $a_2$ and $a_3$, ensuring that there is always at least one positive real root of $Q(x)$. Hence, there is at least a root $\bar{\rho}$ of $P(\rho)$ that is real and satisfies $\bar{\rho}>1$ for all $n\geq 2$, $\lVert A \rVert_{\infty}>0$ and $p\geq 0$.

The same exact argument also holds for Clenshaw-Curtis quadrature and the thesis of Theorem \ref{thm:error_bound} is thus proved. To derive the bounds in Corollary \ref{coroll:error_bound} we start from equation \eqref{eq:thm_proof0}, which we simplify by noting that $g(\rho) \leq \rho$ and $\frac{144}{35}\frac{1}{\rho^2-1}<1$ for $\rho\geq 1+\sqrt{2}$. After minimizing the result with respect to $\rho$, we obtain
\begin{align}
    \label{eq:thm_proof1}
    \lVert I_{p+1} - \hat{I}_{p+1} \rVert_{\infty} \leq \min_{\rho \geq 1 + \sqrt{2}} \frac{\rho^{-2n+p}}{2^{p+1}p!}e^{\frac{1}{2}\rho\lVert A \rVert_{\infty}} \lVert b \rVert_{\infty}=\frac{\lVert b \rVert_{\infty}}{2^{p+1}p!}\left(\frac{2n-p}{\frac{e}{2}\lVert A \rVert_{\infty}}\right)^{-(2n-p)},
\end{align}
which is \eqref{eq:thm_gaussian_result}. Here in the last passage we used the fact that the minimum is attained at $\rho=\max(1+\sqrt{2},\tilde{\rho})$, where $\tilde{\rho}=2\lVert A\rVert^{-1}_{\infty}(2n-p)$. Note that for the expression on the right in \eqref{eq:thm_proof1} to be decreasing in $n$ we need $n > \frac{1}{4}\lVert A \rVert_{\infty} + \frac{p}{2}$, for which $\tilde{\rho}>1$. Taking $n\geq\frac{1+\sqrt{2}}{4}\lVert A \rVert_{\infty} + \frac{p}{2}$ ensures that $\tilde{\rho}\geq 1+\sqrt{2}$ and that the bound \eqref{eq:thm_proof1} holds. For Clenshaw-Curtis quadrature the same simplifications for $\tilde{\rho}\geq 1+\sqrt{2}$ yield the similar result for even $n$:
\begin{align}
    \label{eq:thm_proof2}
    \lVert I_{p+1} - \hat{I}_{p+1} \rVert_{\infty} \leq \min_{\rho \geq
    1+\sqrt{2}} \frac{\rho^{-n+p}}{2^{p+1}p!}e^{\frac{1}{2}\rho\lVert A \rVert_{\infty}} \lVert b \rVert_{\infty}= \frac{\lVert b \rVert_{\infty}}{2^{p+1}p!}\left(\frac{n-p}{\frac{e}{2}\lVert A \rVert_{\infty}}\right)^{-(n-p)},
\end{align}
where $n$ must be replaced with $n-1$ if $n$ is odd. The bound \eqref{eq:thm_proof2} is \eqref{eq:thm_clenshaw_curtis_result}. In this case, the minimum is attained at $\max(1+\sqrt{2},\tilde{\rho})$, where $\tilde{\rho}=2\lVert A\rVert^{-1}_{\infty}(n-p)$, and for the right-hand side expression to be decreasing in $n$ we now need $n>\frac{1}{2}\lVert A\rVert_{\infty} + p$, for which $\tilde{\rho}>1$. Taking $n\geq\frac{1+\sqrt{2}}{2}\lVert A \rVert_{\infty} + p$ ensures that $\tilde{\rho}\geq 1+\sqrt{2}$ and that the bound \eqref{eq:thm_proof2} holds.
\end{proof}

\begin{remark}
Theorem \ref{thm:error_bound} and Corollary \ref{coroll:error_bound} establish that the rate of convergence of the quadrature rules used to approximate the $\varphi$-functions is supergeometric. 
\end{remark}

The numerical approximation of the roots of a polynomial is nowadays a straightforward, robust, fast, and accurate procedure. Theorem \ref{thm:error_bound} thus inspires a definite recipe to compute an upper bound for the quadrature error and for the minimum number of quadrature nodes required to achieve a given error tolerance. We present the related routines in Algorithms \ref{alg:quaderr} and \ref{alg:quadnodes}.

\begin{algorithm}[h!]\label{}
\caption{\texttt{$E=$ quaderr$(n,p,\alpha,\beta)$}.\\\textbf{Description:} Algorithm for estimating the quadrature error.}
\label{alg:quaderr}
\begin{enumerate}[leftmargin=*,align=left]
    \item[\textbf{Input:}] An integer $p$ corresponding to the maximum value of $p$ for which computing $\varphi_p(A)b$ is required. An estimate $\alpha\approx\lVert A \rVert_{\infty}$, $\beta=\lVert b \rVert_{\infty}$, and a chosen number of quadrature nodes $n+1$.
    \begin{itemize}
        \item Set $E = 0$. Then, \textbf{for} $q=1,\dots,p$ repeat:
        \begin{enumerate}
        \item[1)] Use \eqref{eq:polycoeffs} with $p=q-1$ to compute the coefficients of \eqref{eq:thm_polyeq}.
        \item[2)] Solve \eqref{eq:thm_polyeq} and select $\bar{\rho}>1$ to be the root that minimizes either \eqref{eq:thm_bnd_CC} or \eqref{eq:thm_bnd_Gauss} depending on the quadrature rule used. Store the corresponding error bound in $\bar{E}$ and set $E = \max(E, \bar{E})$.
        \end{enumerate}
    \end{itemize}
    \item[\textbf{Output:}] An upper bound $E$ on the quadrature error for computing $\varphi_q(A)b$ valid for all $q=1,\dots,p$. 
\end{enumerate}
\end{algorithm}

\begin{algorithm}[h!]
\caption{\texttt{$n=$ quadnodes$(\varepsilon,p,\alpha,\beta)$}.\\\textbf{Description:} Algorithm for estimating the number of quadrature nodes.}
\label{alg:quadnodes}
\begin{enumerate}[leftmargin=*,align=left]
    \item[\textbf{Input:}] An integer $p$, an estimate $\alpha\approx\lVert A \rVert_{\infty}$, $\beta=\lVert b \rVert_{\infty}$, and a quadrature error tolerance $\varepsilon$.
    \begin{itemize}
        \item Set $n=2$ for Gaussian quadrature and $n=4$ for Clenshaw-Curtis quadrature, and use Algorithm \ref{alg:quaderr} to compute \texttt{$E=$ quaderr$(n,p,\alpha,\beta)$}.
        \item \textbf{While} $E>\varepsilon$: set $n=2n$ and compute the corresponding error \texttt{$E=$ quaderr$(n,p,\alpha,\beta)$}.
        \item Find the unique\footnotemark\ root of $h(n)=$ \texttt{quaderr$(n,p,\alpha,\beta) - \varepsilon$} in $[n/2,\ n]$ via scalar rootfinding. Set $n=\lceil n \rceil$.
    \end{itemize}
    \item[\textbf{Output:}] An upper bound $n+1$ on the minimum number of nodes required to achieve a quadrature error below $\varepsilon$. 
\end{enumerate}
\end{algorithm}


\begin{remark}
In both Algorithms \ref{alg:quaderr} and \ref{alg:quadnodes} it is best to work with the logarithm of the error bounds in \eqref{eq:thm_bnd_CC} and \eqref{eq:thm_bnd_Gauss} to avoid possible issues with the numerical range of floating-point numbers.
\end{remark}

\subsection{Scaling and modified squaring method}
\label{sec:scaling}
The bound in Corollary \ref{coroll:error_bound} is less sharp than that in Theorem \ref{thm:error_bound}, and thus less useful in practice. Nevertheless, it is more informative as it clearly shows that the rate of convergence is supergeometric. Furthermore, its proof suggests that the number of quadrature nodes should scale linearly with the size of $\lVert A \rVert_{\infty}$, a phenomenon that we indeed observe heuristically when using Algorithm \ref{alg:quadnodes} to compute a suitable $n$ for a wide range of matrix sizes (results not shown for brevity). As it is common for the matrix exponential \cite{al2010new,skaflestad2009scaling}, we therefore use a scaling approach to reduce the size of $\lVert A \rVert_{\infty}$.
\footnotetext{Since the bounds \eqref{eq:thm_bnd_CC} and \eqref{eq:thm_bnd_Gauss} are monotonically decreasing in $n$.}

First, we compute $\varphi_p(2^{-l}A)b$ for a suitable integer $l$, and then scale the result back by using the modified squaring algorithm from \cite{skaflestad2009scaling}, namely:
\begin{align}
    \label{eq:squaring}
    \varphi_p(2A)b=\frac{1}{2^{p}}\left(e^A\varphi_p(A)b + \sum_{j=1}^p\frac{1}{(p-j)!}\varphi_j(A)b\right),
\end{align}
where the action of the matrix exponential is computed according to \eqref{KronAb}. Our method is well-suited for evaluating \eqref{eq:squaring}: Algorithms \ref{alg:1} and \ref{alg:2} compute all vectors $\varphi_j(A)b$ for $j=1,\dots,p$ at little extra cost.  We present our scaling and modified squaring strategy in Algorithm \ref{alg:3}, where we employ equation \eqref{eq:squaring} in point 1).

As an example, if we choose the scaling $l$ to be $l= \log_2\left(\frac{e}{2}\lVert A \rVert_{\infty}\right)$, Theorem \ref{thm:error_bound} then yields the following bounds for the scaled problem:
\begin{align}
    \begin{array}{lll}
    \lVert I_{p+1} - \hat{I}_{p+1} \rVert_{\infty} \leq c_p\lVert b \rVert_{\infty}(2n-p)^{-(2n-p)},& \forall n>\frac{1}{2}(p+1), & \text{for Gaussian quadrature},\\
    \lVert I_{p+1} - \hat{I}_{p+1} \rVert_{\infty} \leq c_p\lVert b \rVert_{\infty}(n-p)^{-(n-p)}, & \forall n>p+1, &\text{for Clenshaw-Curtis quadrature},
    \end{array}
\end{align}
where $c_p=(2^{p+1}p!)^{-1}$ and we assumed $n$ is even in the Clenshaw-Curtis rule. For both quadrature rules and $p\leq20$, a quick computation yields that $21$ quadrature nodes are sufficient to reduce the error below $10^{-20}\lVert b \rVert_{\infty}$.

In practice, such a scaling choice may be excessive and lead to a high squaring cost as well as to loss of significant digits. In fact, while \eqref{eq:squaring} was reported in \cite{skaflestad2009scaling} to be resilient to rounding error accumulation, when $A$ is non-normal excessive squaring may still lead to rounding error accumulation similarly as for the matrix exponential \cite{al2010new,higham2005scaling}.

Motivated by these considerations, we thus design an algorithm that helps balancing scaling and computational expense by calculating the optimal scaling factor that minimizes the total cost. The resulting routine is presented in Algorithm \ref{alg:optimalscaling}, where  we rely on Theorem \ref{thm:error_bound} and Algorithms \ref{alg:quaderr} and \ref{alg:quadnodes} to numerically estimate the optimal values of $l$ and $n$. Algorithm \ref{alg:optimalscaling} is based on modelling the total cost of our algorithm as follows: Let $n+1$ is the final number of quadrature nodes used and let $d$ be the spatial dimension. Then our method requires the computation of $dn$ 1D matrix exponentials, and $n+lp$ matrix-vector products as in \eqref{KronAb}. Since the optimal scaling factor depends on the relative cost of these two operations, we take the total cost to be given by $C(n,l,p)=c_1dn + c_2(n+lp)$ for some suitable positive constants $c_1$ and $c_2$ that are architecture-dependent and must be estimated. In the numerical experiments of Section \ref{sec:numerics} we take $c_1=0$ and $c_2=1$ for simplicity.

\begin{algorithm}[h!]
\caption{\texttt{$(\bar{l},\bar{n},\bar{C}) =$ setup\_quadrature$(\varepsilon,p,\alpha, \beta)$}.\\\textbf{Description:} Algorithm for estimating the optimal scaling and number of quadrature nodes.}
\label{alg:optimalscaling}
\begin{enumerate}[leftmargin=*,align=left]
    \item[\textbf{Input:}] An integer $p$, an estimate $\alpha\approx\lVert A \rVert_{\infty}$, and $\beta=\lVert b \rVert_{\infty}$, and a quadrature error tolerance $\varepsilon$.
    \begin{itemize}
        \item Set $l_{\max}= \lceil \log_2(\alpha)\rceil$ and $\bar{n},\bar{l},\bar{C}=\infty$
        \item \textbf{For} $l=l_{\max},l_{\max}-1,\dots,0$:
        \begin{enumerate}
            \item[1)] Set $\alpha_l=2^{-l}\alpha$ and compute \texttt{$n=$ quadnodes$(\varepsilon,p,\alpha_l,\beta)$} (cf.~Algorithm \ref{alg:quadnodes}).
            \item[2)] \textbf{If} $C(n,l,p)<\bar{C}$\textbf{:} set $\bar{C}=C(n,l,p)$, $\bar{l}=l$, and $\bar{n}=n$. \textbf{Else:} \textbf{break}.
        \end{enumerate}
    \end{itemize}
    \item[\textbf{Output:}] The optimal scaling $\bar{l}$, the corresponding number of quadrature nodes $\bar{n}+1$, and the total cost $\bar{C}$ required to achieve a quadrature error below the tolerance $\varepsilon$.
\end{enumerate}
\end{algorithm}

The reason why we can stop searching in Algorithm \ref{alg:optimalscaling} if $C(n,l,p)>\bar{C}$ is that decreasing the scaling factor causes $n$ to monotonically increase. Therefore $C(n,l,p)$ is convex in $l$ and it will start increasing only after $l$ decreases beyond its minimum.

We remark that when the $A^{x,y,z}$ matrices are sparse computing the infinity norm of $A$ can be typically done efficiently. When $A$ is instead dense, it is possible to estimate its infinity norm via the upper bound
\begin{align}
\label{eq:infnorm_upper_bnd}
\lVert A \rVert_{\infty}\leq \lVert A^x \rVert_{\infty} + \lVert A^y \rVert_{\infty} + \lVert A^z \rVert_{\infty},\quad\text{(with }A^z = 0 \text{ in 2D)}.
\end{align}

\begin{algorithm}[h!]
\caption{\texttt{$Y=$ phiquadmv$(p,A^{x,y,z},b,\alpha,\varepsilon,l,$type$)$}.\\\textbf{Description:} Scaling and modified squaring algorithm for computing $\varphi_j(A)b$ for $j=1,\dots,p$.}
\label{alg:3}
\begin{enumerate}[leftmargin=*,align=left]
    \item[\textbf{Input:}] An integer $p$, the matrices $A^{x,y,z}$ and a vector $b$.
    \item[\textbf{Optional input:}] An estimate $\alpha$ of the infinity norm of $A$ (i.e.~$\alpha\approx \lVert A \rVert_{\infty}$, default: use \eqref{eq:infnorm_upper_bnd}). A tolerance $\varepsilon$ (default: $10^{-14}$). A scaling $l$ (will be estimated if not provided). An integer variable \texttt{type} with values $1$ or $2$ depending on whether Algorithm \ref{alg:1} or \ref{alg:2} is to be employed (default: \texttt{type} $=1$).
    \begin{itemize}
        \item If $l$ not provided, set $\beta=\lVert b \rVert_{\infty}$ and use Algorithm \ref{alg:optimalscaling} to compute \texttt{$(l,n,-)=$ setup\_quadrature$(\varepsilon,p,\alpha,\beta)$}.
        \item Depending on the value of $\texttt{type}$, apply either Algorithm \ref{alg:1} (with Gaussian quadrature using $n+1$ nodes), or Algorithm \ref{alg:2} (with Clenshaw-Curtis adaptive quadrature using $\varepsilon$ as tolerance) to $2^{-l}A$ and obtain $y_j=\varphi_j(2^{-l}A)b$ for $j=1,\dots,p$. Store the matrix exponentials $\exp({2^{-l}A^{x,y,z}})$ used in Algorithm \ref{alg:1} or \ref{alg:2}.
        \item \textbf{For} $k = 0,\dots,l-1$:
        \begin{enumerate}[label=\arabic*)]
            \item Compute $\displaystyle{\hat{y}_i = 2^{-i}\left(e^{2^{k-l}A} y_i + \sum_{j=1}^i\frac{1}{(i-j)!}y_j\right)}$ for $i=1,\dots,p$ using \eqref{KronAb} to compute the $e^{2^{k-l}A} y_i$ terms.
            \item Set $y_j = \hat{y}_j$ for $j=1,\dots,p$. If $k<l-1$, compute $\exp({2^{k+1-l}A^{x,y,z}})=\exp({2^{k-l}A^{x,y,z}})^2$.
        \end{enumerate}
    \end{itemize}
    \item[\textbf{Output:}] The matrix $Y$ such that its $j$-th column is given by the product $y_j=\varphi_j(A)b$ for $j=1,\dots,p$.
\end{enumerate}
\end{algorithm}

With Algorithm \ref{alg:3} we have two options: either employ a direct approach using Gaussian quadrature for a fixed number of points (as in Algorithm \ref{alg:1}) determined from Theorem \ref{thm:error_bound} and Algorithm \ref{alg:optimalscaling}, or employ an adaptive strategy with Clenshaw-Curtis (or another nested rule such as Gauss-Kronrod) as in Algorithm \ref{alg:2}. 
The former approach employs Gaussian quadrature which converges faster, but it comes with no error estimation and relies on the upper bound from \eqref{eq:thm_gaussian_result} which may be an over-estimate. On the other hand, the adaptive strategy uses Clenshaw-Curtis quadrature, but it comes with adaptivity which might improve performance. In the next section we test both methods in practice to determine which one is the most efficient.

\section{Numerical results}\label{sec:numerics}

We now compare the performance of our method in terms of computational time and approximation errors with the state-of-the-art MATLAB routine \texttt{expmv$()$} from Higham et al. \cite{al2011computing} for different problems. We use \texttt{phiquadmv$()$} (i.e.~Algorithm \ref{alg:3}) using either Gaussian quadrature with a fixed number of nodes (i.e.~\texttt{type} $=1$) or adaptive Clenshaw-Curtis quadrature (i.e.~\texttt{type} $=2$). We use the open source software library Chebfun\footnote{Available at \url{https://www.chebfun.org/}.} \cite{driscoll2014chebfun} to compute all required quadrature nodes and weights, and the \texttt{tucker.m} routine from the open source software KronPACK\footnote{Available at \url{https://github.com/caliarim/KronPACK}.} from \cite{caliari2022mu} for the 3D tensor operations required by \eqref{KronAb}. We set the tolerance for the quadrature error to the default value (i.e. $\varepsilon=10^{-14}$) and we employ $C(n,l,p)=n+lp$ in Algorithm \ref{alg:optimalscaling} (i.e. $c_1=0$ and $c_2=1$). In this section, we denote these routines with \texttt{phiquadmv\_gauss$()$} and \texttt{phiquadmv\_cc$()$}, respectively.

In all examples we employ a FE semidiscretization in space with piecewise linear functions and a 2-point Lobatto quadrature to obtain diagonal mass matrices. All the experiments were performed using Matlab version \verb|r2021b| using a single computational thread of an Intel i5-8279U chip with 16GB of RAM via the option \verb|-singleCompThread|.

Our main code is available at \url{https://github.com/jmunoz022/phiquadmv} and the routines for reproducing the results presented in this article are available at \url{https://github.com/jmunoz022/phiquadmv_paper}.

\subsection{Problem 1 - Heat equation in 3D}
We consider the 3D heat equation in $\Omega=(0,1)^3$ for $0\leq t\leq T=1$,
$$u_t-\Delta u=f(x,y,z,t),$$
with homogeneous Dirichlet boundary conditions. The matrix $A$ in this case comes from the semidiscretization of the Laplacian operator and is symmetric positive-definite. We consider a uniform spatial discretization using the same number of elements in each spatial direction so that the matrices $A^{x,y,z}$ have the same dimension. 

Here we set the timestep $\tau=1/8$ and we compute the action of $\varphi_p(-\tau A)$ against the vector $b$ obtained by evaluating the function $u_0(x,y,z)=\sin(\pi x)\sin(\pi y)\sin(\pi z)$ at the nodal points. We monitor the following relative error measure for every value of $p$
\begin{equation}\label{RelativeError}
  \frac{\lVert v_p-v^{kron}_p\rVert_\infty}{\lVert v_p\rVert_\infty},
\end{equation}
where $v_p$ are the actions $\varphi_p(-\tau A)b$ computed with  \texttt{expmv()}, and $v^{kron}_p$ the actions computed with either \texttt{phiquadmv\_gauss()} or \texttt{phiquadmv\_cc()}.

Figure \ref{ErrorsLaplace} shows the relative errors (\ref{RelativeError}) for $p=1,\ldots,20$ and different sizes of the matrix $A$. We select a number of $2^r$ elements in each space dimension with $r=4,\ldots,7$. In Table \ref{TimesLaplace} we compare the computational times in seconds required to compute all 20 actions with \texttt{phiquadmv\_gauss()}, \texttt{phiquadmv\_cc()} and routine \texttt{expmv()}. Table \ref{PointsLaplace} shows the number of quadrature points, the scaling factor 
and the total cost of employing both routines. 

We conclude that both \texttt{phiquadmv} routines perform similarly, are accurate (with relative errors below $10^{–12}$), and are orders of magnitude faster than routine \texttt{expmv()}. In particular, for a matrix of size near to 2 million, \texttt{expmv()} routine required 12.5 hours to compute all actions while \texttt{phiquadmv\_gauss()} and \texttt{phiquadmv\_cc()} took only 25 and 31 seconds (1750 and 1424 times faster), respectively.

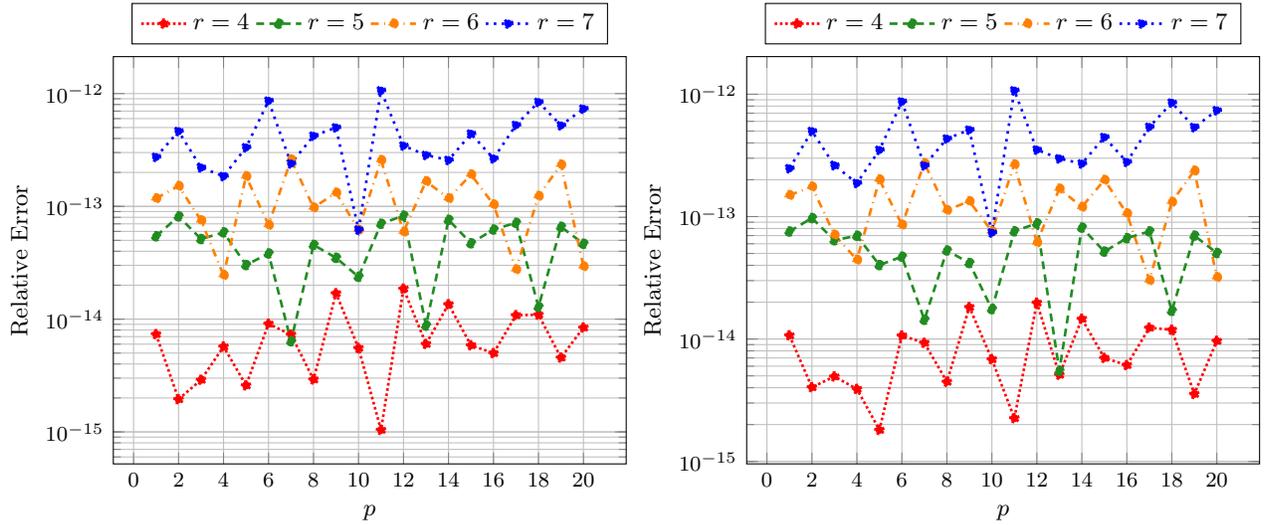
\begin{figure}[h!]
\centering
\begin{tikzpicture}
\begin{axis}[small,
	width=1.2*\width,
	height=\height,
	ymode=log,
	xlabel=$p$,
	ylabel=Relative Error,
	grid=both,
    legend style={font=\small,at={(0.5,1.03)},anchor=south},
    legend entries={$r=4$,$r=5$, $r=6$, $r=7$},
	legend columns=4
	]
\addplot[
	color=red,densely dotted,line width=1pt, 
	mark=*,mark size=1.5pt,    
	]
table[x=q,y=Error_gauss]{Results/Results_Laplace/ErrLaplace_r4_m3.txt};
\addplot[
	color=forestgreen,densely dashed,line width=1pt,
	mark=*,mark size=1.5pt,    
	]
table[x=q,y=Error_gauss]{Results/Results_Laplace/ErrLaplace_r5_m3.txt};
\addplot[
	color=orange,dashdotted,line width=1pt, 
	mark=*,mark size=1.5pt,    
	]
table[x=q,y=Error_gauss]{Results/Results_Laplace/ErrLaplace_r6_m3.txt};
\addplot[
	color=blue,dotted,line width=1pt, 
	mark=*,mark size=1.5pt,    
	]
table[x=q,y=Error_gauss]{Results/Results_Laplace/ErrLaplace_r7_m3.txt};
\end{axis}
\end{tikzpicture}
\begin{tikzpicture}
\begin{axis}[small,
	width=1.2*\width,
	height=\height,
	ymode=log,
	xlabel=$p$,
	ylabel=Relative Error,
	grid=both,
    legend style={font=\small,at={(0.5,1.03)},anchor=south},
    legend entries={$r=4$,$r=5$, $r=6$, $r=7$},
	legend columns=4
	]
\addplot[
	color=red,densely dotted,line width=1pt, 
	mark=*,mark size=1.5pt,    
	]
table[x=q,y=Error_cheb]{Results/Results_Laplace/ErrLaplace_r4_m3.txt};
\addplot[
	color=forestgreen,densely dashed,line width=1pt,
	mark=*,mark size=1.5pt,    
	]
table[x=q,y=Error_cheb]{Results/Results_Laplace/ErrLaplace_r5_m3.txt};
\addplot[
	color=orange,dashdotted,line width=1pt, 
	mark=*,mark size=1.5pt,    
	]
table[x=q,y=Error_cheb]{Results/Results_Laplace/ErrLaplace_r6_m3.txt};
\addplot[
	color=blue,dotted,line width=1pt, 
	mark=*,mark size=1.5pt,    
	]
table[x=q,y=Error_cheb]{Results/Results_Laplace/ErrLaplace_r7_m3.txt};
\end{axis}
\end{tikzpicture}
\caption{Relative error of computing $\varphi_p(-\tau A)b$ with \texttt{phiquadmv\_gauss()} (left) and \texttt{phiquadmv\_cc()} (right) for Problem 1, $p=1,\ldots,20$ and $r=4,\ldots,7$.}
\label{ErrorsLaplace}
\end{figure}

\begin{table}[h!]
\centering
\begin{tabular}{|c|c|c|c|}\hline
Size of A&\texttt{phiquadmv\_gauss()}&\texttt{phiquadmv\_cc()}&\texttt{expmv()}\\\hline
3375&0.19&0.11&1.74\\
29791&0.95&1.15&31.16\\
250047&1.34&1.79&1018.58\\
2048383&25.63&31.50&44855.12\\\hline
\end{tabular}
\caption{Computational time in seconds for computing $\varphi_p(-\tau A)b$ in Problem 1 with $p=1,\ldots,20$.}
\label{TimesLaplace}
\end{table}

\begin{table}[h!]
\centering
\begin{tabular}{|c|c|c|c|c|c|c|c|}\hline
&& \multicolumn{3}{c|}{\texttt{phiquadmv\_gauss()}} &  \multicolumn{3}{c|}{\texttt{phiquadmv\_cc()}}\\\hline
Size of A&$\lVert A\rVert_\infty$&$\bar{n}$&$\bar{C}$&$\bar{l}$&$\bar{n}$&$\bar{C}$&$\bar{l}$\\\hline
3375&384&37&97&3&49&129&4\\
29791&1536&37&137&5&49&169&6\\
250047&6144&37&177&7&49&209&8\\
2048383&24576&37&217&9&49&249&10\\\hline
\end{tabular}
\caption{Number of quadrature nodes $\bar{n}$, scaling factor $\bar{l}$ and cost $\bar{C}$ employed for Problem 1 with $\max p=20$.}
\label{PointsLaplace}
\end{table}

\subsection{Problem 2 - Advection-diffusion problem with a Sishkin mesh}
We now consider the 2D Eriksson-Johnson problem over $\Omega=(-1,0)\times(-0.5,0.5)$ for $0\leq t\leq T=1$ as presented in \cite{munoz2022exploiting}. Here, the matrix $A$ comes from the semidiscretization of the advection-diffusion operator
$$u_t+\gamma\cdot\nabla u-\epsilon\Delta u=f(x,y,t),$$
with both Neumman and Dirichlet boundary conditions
\begin{equation}\nonumber
  \displaystyle{\left\{
      \begin{split}
        &\frac{\partial}{\partial x}u(-1,y,t)=10e^{-4t}\left(y^2-0.25\right)+\frac{r_{1}e^{-r_{1}}-r_{2}e^{-r_{2}}}{e^{-r_{1}}-e^{-r_{2}}}\cos(\pi y),\\
        &u(0,y,t)=u(x,-0.5,t)=u(x,0.5,t)=0,\\
      \end{split}
    \right.} 
\end{equation}
where $\gamma=(1,0)$, $\epsilon=10^{-2}$,
$r_{1,2}=\frac{1\pm\sqrt{4\pi^{2}\epsilon^{2}}}{2\epsilon}$. We now set the vector $b$ with the nodal values of the initial condition $u(x,y,0)=10x\left(y^2-0.25\right)+\frac{e^{r_{1}x}-e^{r_{2}x}}{e^{-r_{1}}-e^{-r_{2}}}\cos(\pi y).$

As in \cite{munoz2022exploiting}, we select a Sishkin mesh (i.e. a graded, piecewise-uniform mesh in the $x$ direction designed to capture the boundary layer, cf.~\cite{kopteva2010shishkin}) with $2^r$ elements in each space dimension. As a consequence of the mesh structure and of the presence of an advection field, the matrices $A$ and $A^{x,y}$ are non-symmetric. Furthermore, $A^{x,y}$ also have different dimensions since we remove the boundary degrees-of-freedom corresponding to the Dirichlet boundary conditions. We again set $\tau=1/8$ and compute $\varphi_p(-\tau A)b$ for $p=1,\dots,20$ and $r=5,\ldots,9$ with both \texttt{phiquadmv\_gauss()} and \texttt{phiquadmv\_cc()}. 

We display in Figure \ref{ErrorsAdvDiff} the relative errors, in Table \ref{TimesAdvDiff} the computational times in seconds and in Table \ref{PointsAdvDiff} the number of quadrature nodes, scaling and total cost of each routine. We conclude that even for this non-symmetric problem, both \texttt{phiquadmv} routines are accurate and faster than \texttt{expmv()}. On the largest matrix, \texttt{expm()} takes 2.18 hours to evaluate the actions while our routines respectively take 7.53 and 11.64 seconds and are 1045 and 676 times faster.

\begin{figure}[h!]
\centering
\begin{tikzpicture}
\begin{axis}[small,
	width=1.17*\width,
	height=\height,
	ymode=log,
	xlabel=$p$,
	ylabel=Relative Error,
	grid=both,
    legend style={font=\small,at={(0.5,1.03)},anchor=south},
    legend entries={$r=5$,$r=6$, $r=7$, $r=8$,$r=9$},
    legend columns=5
	]      
\addplot[
	color=red,densely dotted,line width=1pt, 
	mark=*,mark size=1.5pt,    
	]
table[x=q,y=Error_gauss]{Results/Results_AdvDiff/ErrAdvDiff_r5_m3.txt};
\addplot[
	color=forestgreen,densely dashed,line width=1pt, 
	mark=*,mark size=1.5pt,    
	]
table[x=q,y=Error_gauss]{Results/Results_AdvDiff/ErrAdvDiff_r6_m3.txt};
\addplot[
	color=orange,dashdotted,line width=1pt, 
	mark=*,mark size=1.5pt,    
	]
table[x=q,y=Error_gauss]{Results/Results_AdvDiff/ErrAdvDiff_r7_m3.txt};
\addplot[
	color=blue,dotted,line width=1pt, 
	mark=*,mark size=1.5pt,    
	]
table[x=q,y=Error_gauss]{Results/Results_AdvDiff/ErrAdvDiff_r8_m3.txt};
\addplot[
	color=black,dashed,line width=1pt, 
	mark=*,mark size=1.5pt,    
	]
table[x=q,y=Error_gauss]{Results/Results_AdvDiff/ErrAdvDiff_r9_m3.txt};
\end{axis}
\end{tikzpicture}
\begin{tikzpicture}
\begin{axis}[small,
	width=1.17*\width,
	height=\height,
	ymode=log,
	xlabel=$p$,
	ylabel=Relative Error,
	grid=both,
    legend style={font=\small,at={(0.5,1.03)},anchor=south},
    legend entries={$r=5$,$r=6$, $r=7$, $r=8$,$r=9$},
    legend columns=5
	]      
\addplot[
	color=red,densely dotted,line width=1pt, 
	mark=*,mark size=1.5pt,    
	]
table[x=q,y=Error_cheb]{Results/Results_AdvDiff/ErrAdvDiff_r5_m3.txt};
\addplot[
	color=forestgreen,densely dashed,line width=1pt, 
	mark=*,mark size=1.5pt,    
	]
table[x=q,y=Error_cheb]{Results/Results_AdvDiff/ErrAdvDiff_r6_m3.txt};
\addplot[
	color=orange,dashdotted,line width=1pt, 
	mark=*,mark size=1.5pt,    
	]
table[x=q,y=Error_cheb]{Results/Results_AdvDiff/ErrAdvDiff_r7_m3.txt};
\addplot[
	color=blue,dotted,line width=1pt, 
	mark=*,mark size=1.5pt,    
	]
table[x=q,y=Error_cheb]{Results/Results_AdvDiff/ErrAdvDiff_r8_m3.txt};
\addplot[
	color=black,dashed,line width=1pt, 
	mark=*,mark size=1.5pt,    
	]
table[x=q,y=Error_cheb]{Results/Results_AdvDiff/ErrAdvDiff_r9_m3.txt};
\end{axis}
\end{tikzpicture}
\caption{Relative error of computing $\varphi_p(-\tau A)b$ with \texttt{phiquadmv\_gauss()} (left) and \texttt{phiquadmv\_cc()} (right) for Problem 2, $p=1,\ldots,20$ and $r=5,\ldots,9$.}
\label{ErrorsAdvDiff}
\end{figure}
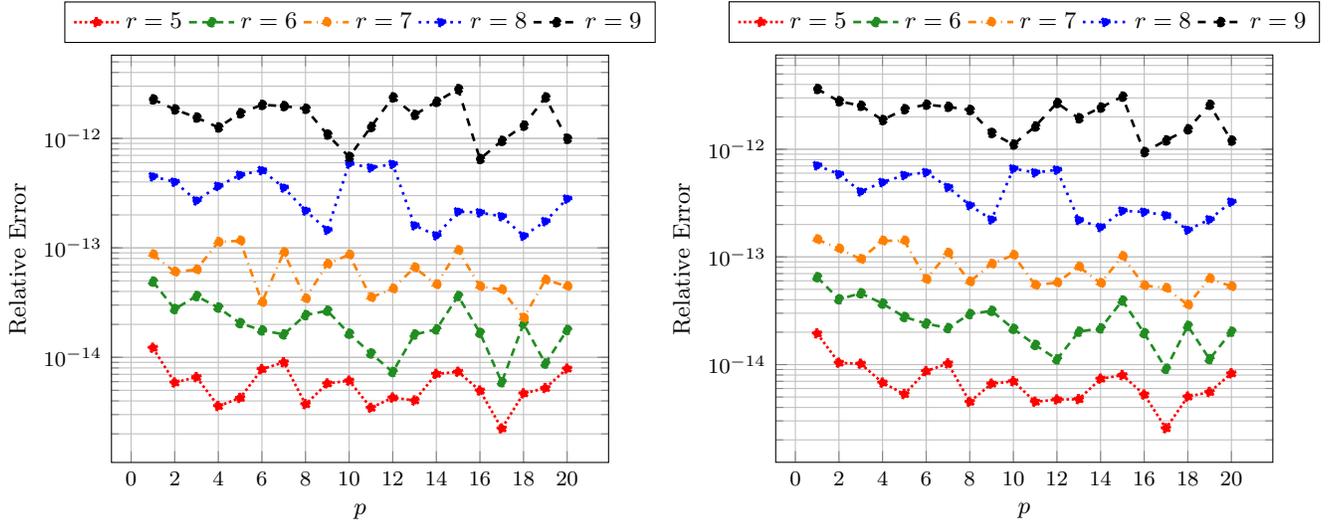

\begin{table}[h!]
\centering
\begin{tabular}{|c|c|c|c|}\hline
Size of A&\texttt{phiquadmv\_gauss()}&\texttt{phiquadmv\_cc()}&\texttt{expmv()}\\\hline
992&0.20&0.20&0.51\\
4032&0.12&0.20&2.03\\
16256&0.24&0.49&30.20\\
65280&1.30&2.248&461.41\\
261632&7.53&11.64&7869.39\\\hline
\end{tabular}
\caption{Computational times in seconds for computing $\varphi_p(-\tau A)b$ in Problem 2 with $p=1,\ldots,20$.}
\label{TimesAdvDiff}
\end{table}

\begin{table}[h!]
\centering
\begin{tabular}{|c|c|c|c|c|c|c|c|}\hline
&& \multicolumn{3}{c|}{\texttt{phiquadmv\_gauss()}} &  \multicolumn{3}{c|}{\texttt{phiquadmv\_cc()}}\\\hline
Size of A&$\lVert A\rVert_\infty$&$\bar{n}$&$\bar{C}$&$\bar{l}$&$\bar{n}$&$\bar{C}$&$\bar{l}$\\\hline
992&332.8&34&94&3&97&177&4\\
4032&1331.2&34&134&5&97&217&6\\
16256&5324.8&34&174&7&97&257&8\\
65280&21299.2&34&214&9&97&297&10\\
261632&85196.8&34&254&11&97&337&12\\\hline
\end{tabular}

\caption{Number of quadrature nodes $\bar{n}$, scaling factor $\bar{l}$ and cost $\bar{C}$ employed for Problem 2 with $\max p=20$.}
\label{PointsAdvDiff}
\end{table}

\begin{remark}
We note that in Figures \ref{ErrorsLaplace} and \ref{ErrorsAdvDiff} the approximation error is small, yet above the prescribed tolerance of $\varepsilon=10^{-14}$. Even assuming that the \texttt{expmv()} routine is exact, this behavior is likely a consequence of rounding errors, which our analysis from Section \ref{sec:algorithm} does not account for. In particular, independently from the scaling factor used, we cannot expect to reduce the error below the condition number of the problem times the unit roundoff of the floating-point format used. While the condition number of computing $\varphi$-functions of matrices has not, to the best of our knowledge, been investigated, we know for instance that for the matrix exponential (cf.~Lemma 10.15 in \cite{higham2008functions}) this is at least as big as $\lVert A \rVert_{\infty}$. Looking at the size of $\lVert A \rVert_{\infty}$ in Tables \ref{PointsLaplace} and \ref{PointsAdvDiff}, we can then expect to lose a few digits in our computations.
\end{remark}

\subsection{Problem 3 - Hochbruck-Ostermann equation}
We consider the semilinear Hochbruch-Ostermann equation from \cite{hochbruck2005explicit} over $\Omega=(0,1)^2$ and $0\leq t\leq T=1$
$$u_t-\Delta u=\frac{1}{1+u^2}+f(x,y,t),$$
subject to homogeneous Dirichlet boundary conditions. Here, we select the linear source $f$ and the initial condition $u_0$ using the method of manufactured solutions in such a way that the exact solution is $u(x,y,t)=x(1-x)y(1-y)e^t$. 

We compare the performance of our algorithm with three exponential Runge-Kutta methods from \cite{hochbruck2005explicit} defined by the Butcher tableaus in Table \ref{RK} (in which we denote $\varphi_{i,j}:=\varphi_{i}(-c_j\tau A)$). We select $c_2=\frac{1}{2}$ in the two-stage Runge-Kutta method and $c_2=\frac{1}{3}$ in the three-stage one. In Figure \ref{ConvergenceHochOster} we show the errors of the approximations obtained by both routines \texttt{phiquadmv\_gauss()} and \texttt{phiquadmv\_cc()} for the three Runge-Kutta methods at the final time $T=1$ (both routines deliver the same convergence results). Here, we work with a fixed mesh with $2^{10}$ elements in each space dimension and we monitor the error behaviour in the infinity norm. We observe the expected order of convergence in time as we refine the time-step up until the error in space becomes dominant, showing that our method is accurate and does not affect the convergence of the exponential integrators.

\begin{table}[h!]
\begin{minipage}{.25\linewidth}
\centering
\begin{tabular}{c|c}
0&\\\hline
&$\varphi_1$
\end{tabular}
\caption*{Exponential Euler}
\end{minipage}
\begin{minipage}{.3\linewidth}
\centering
\begin{tabular}{c|cc}
0&&\\
$c_2$&$c_2\varphi_{1,2}$&\\\hline
&$(1-\frac{1}{2c_2})\varphi_1$&$\frac{1}{2c_2}\varphi_1$
\end{tabular}
\caption*{RK2}
\end{minipage}
\begin{minipage}{.45\linewidth}
\centering
\begin{tabular}{c|ccc}
0&&&\\
$c_2$&$c_2\varphi_{1,2}$&&\\
$\frac{2}{3}$&$\frac{2}{3}\varphi_{1,3}-\frac{4}{9c_2}\varphi_{2,3}$&$\frac{4}{9c_2}\varphi_{2,3}$&\\\hline
&$\varphi_1-\frac{3}{2}\varphi_2$&0&$\frac{3}{2}\varphi_2$
\end{tabular}
\caption*{RK3}
\end{minipage}
\caption{Butcher tableaus corresponding to Exponential Runge-Kutta methods up to order 3. Here $\varphi_{i,j}:=\varphi_{i}(-c_j\tau A)$.}
\label{RK}
\end{table}

\begin{figure}[h!]
\centering
\begin{tikzpicture}
\begin{axis}[small,
	width=1.45*\width,
	height=\height,
	ymode=log,
	xmode=log,
	xlabel=$\tau$,
	ylabel=Error at final time,
	grid=both,
    legend style={font=\small,at={(0.5,1.03)},anchor=south},
    legend entries={Euler,RK2,RK3,Rate 1, Rate 2, Rate 3},
    legend columns=3
	]      
\addplot[
	color=red,dashdotted,line width=1pt, 
	mark=*,mark size=1.5pt,    
	]
table[x=tau,y=Err_EulerKron]{Results/Results_HochOster/ErrHochOster_Kron_r10_m10.txt};
\addplot[
	color=blue,dotted,line width=1pt, 
	mark=*,mark size=1.5pt,    
	]
table[x=tau,y=Err_RK2Kron]{Results/Results_HochOster/ErrHochOster_Kron_r10_m10.txt};
\addplot[
	color=black,dashed,line width=1pt, 
	mark=*,mark size=1.5pt,    
	]
table[x=tau,y=Err_RK3Kron]{Results/Results_HochOster/ErrHochOster_Kron_r10_m10.txt};
\addplot[
	color=red,solid,line width=1pt, 
	]
table[x=tau,y=Slope1]{Results/Results_HochOster/Slope1.txt};
\addplot[
	color=blue,solid,line width=1pt, 
	]
table[x=tau,y=Slope2]{Results/Results_HochOster/Slope2.txt};
\addplot[
	color=black,solid,line width=1pt, 
	]
table[x=tau,y=Slope3]{Results/Results_HochOster/Slope3.txt};
\end{axis}
\end{tikzpicture}
\caption{Convergence in time at $T=1$ of Runge-Kutta methods up to order 3 for Problem 3 computed with both routines \texttt{phiquadmv\_gauss()} and \texttt{phiquadmv\_cc()}. The space mesh is fixed with $2^{10}$ elements in each dimension.}
\label{ConvergenceHochOster}
\end{figure}
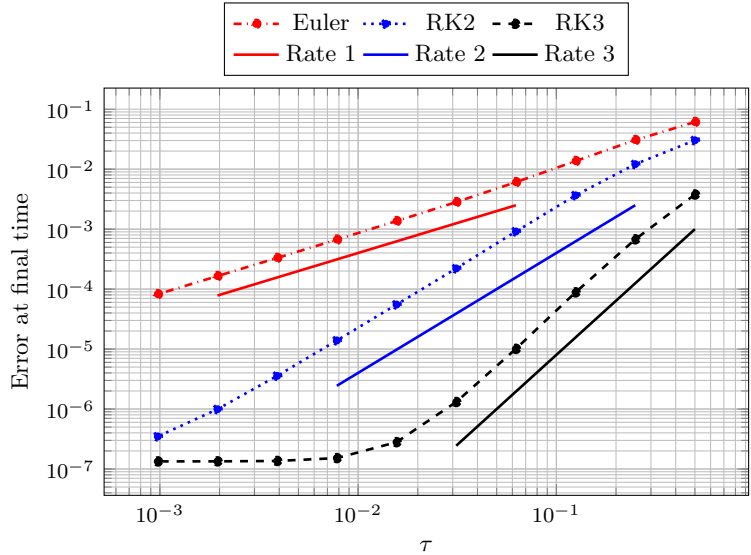

We now compare the efficiency of the methods \texttt{phiquadmv\_gauss()}, \texttt{phiquadmv\_cc()}, and \texttt{expmv()} when used in conjunction with exponential integrators to solve the Hochbruck-Ostermann equation. In Tables \ref{TimesHochOster7} and \ref{TimesHochOster8} we record the total CPU time spent by these routines for different number of time step sizes and for exponential Runge-Kutta methods of order up to $3$. We compare two discretizations in space, fixing $2^7$ and $2^8$ elements per spatial direction, respectively.

We conclude that in all cases \texttt{phiquadmv\_gauss()} and \texttt{phiquadmv\_cc()} accelerate the computation of the exponential time integrators compared to \texttt{expmv()}. Nevertheless, we observe that the growth of the computational time for \texttt{expmv()} is slower as we refine the time step size for a fixed discretization in space. Therefore, the largest gain we obtain with the \texttt{phiquadmv()} routines is when the the time step size is large compared to the discretization in space. Also, we observe that in this case \texttt{phiquadmv\_gauss()} is faster than \texttt{phiquadmv\_cc()} by a factor between two and three, which is consistent with the results from Section \ref{sec:algorithm}.

\begin{remark}
We remark that even though the results presented in this section have been obtained in serial our methods are well-suited for parallelism since computations at different quadrature nodes as well as the squaring of $\varphi_p(A)b$ for different $p$ can be performed independently. We leave a parallel implementation of our routines to future work.
\end{remark}

\begin{table}[h!]
\begin{adjustbox}{width=1\textwidth}
\centering
\begin{tabular}{|c|c|c|c|c|c|c|c|c|c|}\hline
& \multicolumn{3}{c|}{Euler} &  \multicolumn{3}{c|}{RK2} & \multicolumn{3}{c|}{RK3}\\\hline
Time steps&\texttt{phiquadmv\_gauss()}&\texttt{phiquadmv\_cc()}&\texttt{expmv()}&\texttt{phiquadmv\_gauss()}&\texttt{phiquadmv\_cc()}&\texttt{expmv()}&\texttt{phiquadmv\_gauss()}&\texttt{phiquadmv\_cc()}&\texttt{expmv()}\\\hline
2&0.14&0.17&39.26&0.10&0.15&60.24&0.18&0.41&84.95\\
4&0.09&0.12&36.88&0.12&0.22&56.44&0.27&0.76&82.27\\
8&0.12&0.21&37.62&0.22&0.40&56.85&0.57&1.37&83.96\\
16&0.26&0.40&38.40&0.38&0.75&58.18&1.19&2.67&84.94\\
32&0.42&0.77&39.74&0.67&1.55&60.85&1.95&5.19&89.42\\
64&0.65&1.47&41.61&1.24&2.88&64.58&3.47&10.27&94.10\\
128&1.23&2.88&45.26&2.42&5.69&70.63&6.56&20.29&103.34\\
256&2.34&5.64&49.92&4.64&11.20&78.95&12.83&40.26&116.12\\
512&4.55&11.15&58.88&8.94&22.22&95.65&25.17&78.17&146.51\\\hline
\end{tabular}
\end{adjustbox}
\caption{Computational times in seconds Problem 3 for different number of time step sizes. The mesh in space is fixed to $2^7$ elements per space dimension.}
\label{TimesHochOster7}
\end{table}

\begin{table}[h!]
\begin{adjustbox}{width=1\textwidth}
\centering
\begin{tabular}{|c|c|c|c|c|c|c|c|c|c|}\hline
& \multicolumn{3}{c|}{Euler} &  \multicolumn{3}{c|}{RK2} & \multicolumn{3}{c|}{RK3}\\\hline
Time steps&\texttt{phiquadmv\_gauss()}&\texttt{phiquadmv\_cc()}&\texttt{expmv()}&\texttt{phiquadmv\_gauss()}&\texttt{phiquadmv\_cc()}&\texttt{expmv()}&\texttt{phiquadmv\_gauss()}&\texttt{phiquadmv\_cc()}&\texttt{expmv()}\\\hline
2&0.13&0.25&664.50&0.22&0.41&970.01&0.59&1.37&1422.08\\
4&0.22&0.41&676.49&0.40&0.79&973.64&1.05&2.66&1454.27\\
8&0.38&0.78&697.32&0.81&1.50&1007.13&2.10&5.20&1434.99\\
16&0.71&1.47&707.56&1.35&2.90&1045.10&3.93&10.25&1446.83\\
32&1.32&2.87&667.33&2.56&5.76&1014.69&7.74&20.65&1494.18\\
64&2.53&5.65&707.74&5.18&11.31&1056.51&15.06&38.96&1569.51\\
128&4.87&11.07&730.95&10.07&22.51&1111.66&28.89&80.64&1649.21\\
256&9.64&22.56&825.36&19.70&44.81&1226.68&55.00&156.86&1770.38\\
512&19.20&45.55&827.40&38.25&91.02&1277.71&108.33&305.07&1915.17\\\hline
\end{tabular}
\end{adjustbox}
\caption{Computational times in seconds Problem 3 for different number of time step sizes. The mesh in space is fixed to $2^8$ elements per space dimension.}
\label{TimesHochOster8}
\end{table}

\section{Conclusions}\label{sec:coclusions}
We proposed a method that efficiently approximates the action of $\varphi$-functions of matrices with Kronecker sum structure. The algorithm is based on approximating the integral definition of the $\varphi$-functions via either adaptive or fixed-point quadrature combined with a scaling and modified squaring approach. The quadrature rule exploits the Kronecker structure of the matrix and only involves actions of 1D matrix exponentials which can be applied efficiently. Evaluation at different quadrature nodes can furthermore be performed in parallel. Additionally, we provided an \textit{a priori} estimate for the quadrature error which shows that our method converges supergeometrically fast with respect to the number of quadrature nodes. Guided by this result, we also designed a strategy for computing the optimal scaling and number of quadrature points that minimizes the total cost while observing a prescribed error tolerance. Numerical experimentation with 2D/3D time-dependent problems with tensor product structure shows that the new method is accurate, efficient and robust, and is well-suited to be combined with exponential integrators. A comparison with the \texttt{expmv()} state-of-the-art routine from Al-Mohy and Higham revealed that for matrices with Kronecker sum structure our method can accelerate the computation of the actions of $\varphi$-matrix-functions by orders of magnitude.

Possible extensions of this work include: (1) The extension of our method to linear combinations of the actions of different $\varphi$-functions against different vectors (2) A parallel and/or GPU implementation of the algorithm (3) The application of our technique to spatial semidiscretizations  with IGA for which the 1D matrices are dense. 

\section*{Acknowledgements}
Matteo Croci's work is supported by the Department of Energy, NNSA under Award Number DE-NA0003969.
Judit Muñoz-Matute has received funding from the European Union’s Horizon 2020 research and innovation program under the Marie Sklodowska-Curie individual fellowship No.~101017984 (GEODPG).
 
\bibliography{mybibfile}

\end{document}